\documentclass[12pt]{amsart}
 \usepackage[dvips]{epsfig}
 \usepackage{comment}
 \usepackage{enumerate}
 \usepackage{amsgen, amstext,amsbsy,amsopn, amsthm, amsfonts,amssymb,amscd,amsmath,euscript} %,enumerate,url,verbatim,calc,xypic}
\usepackage{amsmath}
 \makeatother 
\usepackage[notcite,notref]{showkeys}
\newcommand{\edit}[1]{\marginpar{\footnotesize{#1}}}
 \usepackage{latexsym}
 \usepackage{graphics}

\newcommand{\Spec}{\operatorname{Spec}}
\newcommand{\Pic}{{\rm Pic\,}}
\newcommand{\Nil}{{\rm Nil\,}}
\newcommand{\End}{{\rm End}}
\newcommand{\cK}{{\mathcal K}}
\newcommand{\cI}{{\mathcal I}}
\newcommand{\cNI}{{\mathcal{NI}}}

\newcommand{\cO}{{\mathcal O}}

\newcommand{\A}{{\mathbb A}}
\newcommand{\N}{{\mathbb N}}
\newcommand{\Z}{{\mathbb Z}}
\newcommand{\cc}{{\mathfrak c}} 
\def\oo{\otimes}
\def\map#1{{\;\buildrel #1 \over \longrightarrow}\;}
\def\smap#1{{\;\buildrel #1 \over \rightarrow}\;}
\newcommand{\mathdot}{{\mathbf{\scriptscriptstyle\bullet}}}
\newcommand\cof{\rightarrowtail}
\newcommand\bP{{\bf P}}
\newcommand\bNil{{\bf Nil}}
\newcommand\bEnd{{\bf End}}

\newcommand{\et}{{\mathrm{et}}}
\newcommand{\nis}{{\mathrm{nis}}}
\newcommand{\zar}{{\mathrm{zar}}}
\newcommand{\cdh}{{\mathrm{cdh}}}

\theoremstyle{plain}

\newtheorem{theorem}[equation]{Theorem}
\newtheorem{corollary}[equation]{Corollary}
\newtheorem{lemma}[equation]{Lemma}
\newtheorem{substuff}{\bf Remark}[equation]

\newtheorem{proposition}[equation]{Proposition}

\theoremstyle{remark}
\newtheorem{remark}[equation]{Remark}
\newtheorem*{remm}{Remark}

\newtheorem{subrem}[substuff]{Remark} %for numbering x.y.1 etc

\theoremstyle{definition}

 \newtheorem{example}[equation]{Example}
 \newtheorem{fact}{Fact}
 \newtheorem{discussion}{Discussion}

\numberwithin{equation}{section}
\usepackage[notcite,notref]{showkeys}

\begin{document}
\title{Relative Cartier divisors and K-theory}
\date{\today}

 \author{Vivek Sadhu and Charles Weibel}
 \address{School of Mathematics, Tata Institute of Fundamental Research, 
1 Homi Bhaba Road, Colaba, Mumbai 400005, India} 
\email{sadhu@math.tifr.res.in, viveksadhu@gmail.com}
\thanks{Sadhu was supported by TIFR, Mumbai Postdoctoral Fellowship.}
 \address{Math. Dept., Rutgers University, New Brunswick, NJ 08901, USA}
 \email{weibel@math.rutgers.edu}

\begin{abstract}
We study the relative Picard group $\Pic(f)$ of a map $f:X\to S$ of schemes.
If $f$ is faithful affine, it is the relative Cartier divisor group $\cI(f)$.
%is the group $H^0(S,f_*\cO_X^\times/\cO_S^\times)$.
The relative group $K_0(f)$ has a $\gamma$-filtration, and $\Pic(f)$ is
the top quotient for the $\gamma$-filtration. When $f$ is induced by a
ring homomorphism $A\to B$, we show that the relative ``nil'' groups 
$N\Pic(f)$ and $NK_n(f)$ are continuous $W(A)$-modules.  
\end{abstract}
%
%Let $f:A\to B$ be a ring map. In this article, we study the groups
% $N\cI(A,B)$ and $K_{n}(f).$ We prove that $N\cI(f: A\hookr%ightarrow B)$ 
%and $NK_{n}(f)$ are continuous $W(A)$-modules. We also give a
%cohomological interpretation of the groups $N\cI(f)$ and $K_{-n}(f).$

 \maketitle
 
 \section*{Introduction}
If $f:X\to S$ is a morphism of schemes, the relative Picard group
$\Pic(f)$ was defined by Bass in \cite{Bass-Tata}, and
fits into a natural exact sequence
\begin{equation}\label{seq:Pic}
\cO^\times(S)\map{f^*} \cO^\times(X)\map{\partial} \Pic(f) 
\map{} \Pic(S) \map{f^*} \Pic(X).
\end{equation}
The goal of this paper is to study this group as well as $N\Pic(f)$,
defined to be $\Pic(f[t])/\Pic(f)$, where 
$f[t]:X\times\A^1\to S\times\A^1$.

Our first observation is that when $f$ is $\Spec(B)\to\Spec(A)$ for a 
commutative ring extension $A\hookrightarrow B$, $\Pic(f)$
is isomorphic to the relative Cartier divisor group $\cI(f)$,
defined in \cite{rs} as the group of invertible $A$-submodules
of $B$ under multiplication and studied in \cite{ss, vs, SW}.
This definition of $\cI(f)$ also makes sense (and we still have
$\cI(f)\cong\Pic(f)$) for scheme maps $f:X\to S$ for which 
$\cO_S^\times\to f_*\cO_X^\times$ is an injection of sheaves.
It then follows from \cite{SW} that $\Pic(f)$ is a contracted
functor in the sense of Bass.

We then relate $\Pic(f)$ to the relative group $K_0(f)$, which
fits into an exact sequence
\[
K_1(S) \map{f^*} K_1(X) \map{\partial} K_0(f) \map{} K_0(S) \map{} K_0(X).
\]
For example, if $f: A\hookrightarrow B$ is subintegral
then $K_0(f)\cong \Pic(f)$ (Proposition \ref{subint}).

Let $\cNI$ denote the Zariski sheaf associated to the
%etale sheafification of the
presheaf $U\mapsto N\Pic(U,f^{-1}U)$ on $S$. 
In Theorem \ref{H^0} and Theorem \ref{quasi}, we prove the following:

\begin{theorem}
Let $f: X\to S$ be a faithful affine morphism of schemes. 
\begin{enumerate}
\item The Zariski sheaf $\cNI$ is an \'etale sheaf on $S$. Moreover, 
\[
N\Pic(f) \cong H_\et^{0}(S,\cNI)= H_\zar^0(S,\cNI).
\]
 \item If $X$ and $S$ are schemes then 
    $H_\et^*(S,\cNI)\cong H_\zar^*(S,\cNI).$
 \item If $X$ and $S$ are both affine schemes then
$H_\et^q(S,\cNI)=0$ for $q\ne0$.
%\[
%H_\et^q(\Spec(A),\cNI)=\begin{cases}
%                    N\Pic(f) & {\rm if}~ q=0\\
%                    0 & {\rm if}~ q> 0.  \end{cases}
%\]
\end{enumerate}
 \end{theorem}

 A secondary goal of this article is to study the relative
$K$-theory groups $K_{n}(f)$ associated to a morphism of schemes 
$f: X\to S$. By definition, $K_n(f)= \pi_nK(f)$, where $K(f)$ is 
the homotopy fiber of  $K(S)\to K(X).$ Comparing $X\to S$ to
$X[t]\to S[t]$ yields groups $NK_*(f)$.

\begin{theorem}
For each homomorphism $f:A\to B$:
 \begin{enumerate}
  \item $NK_n(f)$ is a continuous $W(A)$-module, for all $n$.
  \item $N\Pic(f)$ is a continuous $W(A)$-module.
  \item $\det: NK_0(f)\to N\Pic(f)$ is a $W(A)$-module homomorphism.
 \end{enumerate}
\end{theorem}

\noindent(See Theorems \ref{NI} and \ref{NK}, and Proposition \ref{ctn-W}).
This implies that if ${\rm char}(A)=p>0$ %and $f:A\hookrightarrow B$ 
then both $NK_{n}(f)$ and $N\Pic(f)$ are $p$-groups, while if
${\rm char}(A)=0$ the groups have the structure of $A$-modules. 
%This follows from the fact that if ${\rm ch}(A)=p>0$ then
% any continuous $W(A)$-module is a $p$-group. In other words,
% we show that $NK_{n}(f)$ and $N\cI(f)$
%are naturally continuous $W(A)$-module. More precisely 

%$\Nil_{n}(f)\cong NK_{n+1}(f)$.

We conclude with some remarks about $K_n(f)$ when $n$ is negative.
If $X$ and $S$ have dimension at most $d$, then $K_n(S)=K_n(X)=0$
for $n<-d$ in many cases.  In such cases, it follows that $K_n(f)=0$
for $n<-d-1$.
% relative K-theory, for details we refer \cite{WK}. 
The cohomological interpretation of the negative $K$-theory of a scheme 
in terms of the $\cdh$-cohomology of the constant sheaf $\Z$ is given in
 \cite{CHSW}. In the relative situation, we prove the following
 (Theorem \ref{vanishing} and Theorem \ref{1dim}):

\begin{theorem}
Let $f:X\to S$ be a finite morphism of $d$-dimensional noetherian schemes.
%essentially of finite type over a noetherian scheme.
\begin{enumerate}
\item If $X$ and $S$ are essentially of finite type over a
  field $k$ of characteristic $0$, $K_{-d-1}(f)\cong H_\cdh^{d}(S,f_{*}\Z/\Z)$.
\item %Let $f:X\to S$ be a finite map of noetherian schemes with
   If $\dim S=1$, then $K_{-2}(f)\cong H_\nis^{1}(S,f_{*}\Z/\Z)$
and there is an extension
\[
0 \to H_\nis^1(S,f_*\cO_X^\times/\cO_S^\times) \to K_{-1}(f) \to 
H_\nis^0(S,f_*\Z/\Z) \to 0.
\]
 \end{enumerate}
\end{theorem}

\section{Relative $\Pic$ and invertible submodules}

In \cite{Bass-Tata}, Bass defined $\Pic(f)$ to be the abelian group 
generated by $[L_1,\alpha,L_2]$, where the $L_i$ are 
line bundles on $S$ and $\alpha:f^*L_1\to f^*L_2$ is an isomorphism.
The relations are:
\begin{enumerate}
\item $[L_1,\alpha,L_2] + [L'_1,\alpha',L'_2] = 
[L_1\oo L'_1,\alpha\oo\alpha',L_2\oo L'_2]$;
\item $[L_1,\alpha,L_2] + [L_2,\beta,L_3] = [L_1,\beta\alpha,L_3]$;
\item $[L_1,\alpha,L_2]=0$ if $\alpha=f^*(\alpha_0)$ 
for some $\alpha_0:L_1\cong L_2$.
\end{enumerate}

\begin{subrem}\label{[L,a]}
By (1), every element of $\Pic(f)$ has the form $[L,\alpha,\cO_S]$.
Writing $[L,\alpha]$ for $[L,\alpha,\cO_S]$, an
alternative presentation for $\Pic(f)$ is that it is generated
by elements $[L,\alpha]$ satisfying:
$[L,\alpha] + [L',\alpha'] = [L\oo L',\alpha\oo\alpha']$;
$[L,\alpha]=0$ if (and only if) there is an isomorphism
$\alpha_0:L\cong\cO_S$ so that $\alpha=f^*(\alpha_0)$.
It is easy to see, and observed by Bass, that the map
$\Pic(f)\to\Pic(S)$ sending $[L,\alpha]$ to $[L]$ 
fits into an exact sequence \eqref{seq:Pic}, where
$\partial(b)=[\cO_S,b]$.
\end{subrem}

\begin{proposition}\label{hyper}
Bass'\hspace{3pt} $\Pic(f)$ is the hypercohomology group 
$H^0(S,\cO_S^\times\!\to\!f_*\cO_X^\times)$.
\end{proposition}

\begin{proof}
Let $C^*$ denote the mapping cone of $\cO_S^\times\!\to\!f_*\cO_X^\times$.
A 0-cocyle of $C^*$ is given by a cover
$\{U_i\}$ of $S$, a unit $b_i$ of $f^{-1}(U_i)$ for each $i$, and units
$a_{ij}$ of $U_i\cap U_j$ for each $i,j$ satisfying the cocyle condition
(so that the $\{a_{ij}\}$ define a line bundle $L$ on $S$) and such that
$b_i/b_j=f^\#(a_{ij})$ on each $f^{-1}(U_i\cap U_j)$. Since the $\{b_i\}$
define an isomorphism $f^*L\cong\cO_X$, each 0-cocyle defines an element
$\lambda=[L,\beta,\cO_S]$ of $\Pic(f)$.  A 0-coboundary is given by 
$a_{ij}=a_i/a_j$ and $b_i=f^\#(a_i)$ for units $a_i$ of $U_i$; adding it to
a cocyle does not change $\lambda$.  Refining the cover does not change 
$\lambda$ either. The result follows from the 5-lemma applied to the following
diagram with exact rows (which is easily checked to be commutative):
\footnotesize
\[\begin{CD}
H^0(S,\cO^\times) @>>> H^0(X,\cO^\times) @>>> H^0(S,C^*) @>>>
%H^0(S,\cO_S^\times\to f_*\cO_X^\times) @>>>
H^1(S,\cO^\times) @>>> H^1(X,\cO^\times)\qquad  \\ 
@V{\cong}VV  @V{\cong}VV  @VVV  @V{\cong}VV  @V{\cong}VV \\
\cO^\times(S) @>>> \cO^\times(X) @>>>\Pic(f) @>>> \Pic(S) @>>> \Pic(X).
\qedhere\end{CD}\]
\normalsize
\end{proof}

\smallskip
Now suppose that $f$ is faithful and affine.
As observed in \cite{SW}, $\cI(f)$ is isomorphic to 
$H^0(S,f_*\cO_X^\times/\cO_S^\times)$. Thus Proposition \ref{hyper}
implies that $\cI(f)\cong\Pic(f)$.  
Here is a more elementary proof.

%It is easy to see, and observed by Bass, that $\Pic(f)$ fits into an %
%exact sequence \eqref{seq:Pic}, where $\Pic(f)\to\Pic(S)$ 
%sends $[L,\alpha]$ to $[L]$ and $\partial(b)=[\cO_S,b]$.
%Note that $[L_1,\alpha,L_2]=[L_2^{-1}\oo L_1,\alpha,\cO_S]$ in $\Pic(f)$.

\begin{lemma}\label{Pic=I}
If $f:X\to S$ is a faithful affine map, there is an isomorphism 
$\rho:\cI(f)\!\map{\cong}\!\Pic(f)$, 
sending $L$ to $[L,i,\cO_S]$, where
$i:f^*L\cong\cO_X$.
%$i:L\hookrightarrow f_*\cO_X$ to $[L,i,\cO_S]$.
\end{lemma}

The isomorphism $f^*L\cong\cO_X$ is well defined, because in any affine 
open $U=\Spec(A)$ of $S$ we have $f^{-1}U=\Spec(B)$ with $A\subset B$; 
it was proven by Roberts and Singh \cite{rs} that
$L\subset B$ induces $L\oo_AB \cong B$.

\begin{proof}
Since $\rho(LL')=[L\oo L',i\oo i',\cO_S]=[L,i,\cO_S]+[L',i',\cO_S]$,
$\rho$ is a homomorphism.  To define the inverse map, we use the 
presentation of $\Pic(f)$ and the observation that because
$\cO_S\to f_*\cO_X$ is an injection, so is $L\to L\oo f_*\cO_X$
for every line bundle $L$. Given a triple $[L_1,\alpha,L_2]$, we set
$L=L_2^{-1}\oo L_1$, so that $\alpha$ induces an isomorphism
$f^*L\cong f^*(L_2)^{-1}\oo f^*(L_1)\cong \cO_X$, 
and define $\psi([L_1,\alpha,L_2])$ to be the 
submodule $L$ of $L\oo f_*\cO_X\cong f_*\cO_X$.
Since $\psi$ is compatible with
the relations of $\Pic(f)$, it descends to a homomorphism
$\psi:\Pic(f)\to\cI(f)$.  Since 
$[L_1,\alpha,L_2]=[L_2^{-1}\oo L_1,\alpha,\cO_S]$ in $\Pic(f)$
and $f^*(L)=\cO_X$ for all $L\in\cI(f)$, $\psi$ is an inverse to $\rho$.
\end{proof}

%Since $\Pic(S)$ is isomorphic to $H^1(S,\cO_S^\times)$ and $\cI(f)$ is
%isomorphic to $H^0(S,f_*\cO_X^\times/\cO_S^\times)$,

\begin{comment}
If $f:X\to S$ is any morphism, the sheaf of $\cO_S$-algebras $f_*\cO_X$
defines a scheme $X_0=\Spec(f_*\cO_X)$ such that $f_0:X_0\to S$ is affine
and $f$ factors as $X\to X_0\map{f_0} S$.  (See \cite[Ex.\,II.5.17]{Hart}.)
The following lemma reduces us to the case when $f$ is an affine map.

\begin{lemma}
$\Pic(f)\cong\Pic(f_0)$.
\end{lemma}

\begin{proof}
From the Leray spectral sequence for $\cO_X^\times$ and $X\to X_0$, we see that
$\Pic(X_0)\to\Pic(X)$ is an injection. The result is now
immediate from \eqref{seq:Pic} since $\cO^\times(X)\cong\cO^\times(X_0)$.
\end{proof}
\end{comment}

\section{Relative $K_0$ and $\Pic$}

Bass gave a presentation of a relative group $K_0(f)$ associated to
$f:A\to B$ in \cite{Bass-Tata} and \cite[VII.5]{b};
see \cite[II.2.10]{WK}. 
It is generated by triples $[P_1,\alpha,P_2]$, where the 
$P_i$ are finitely generated projective $A$-modules (or vector bundles 
on $S$) and $\alpha$ is an isomorphism $f^*(P_1)\map{\cong}f^*(P_2)$,
and agrees with the group $\pi_0K(f)$ of \cite[IV.1.11]{WK}.
The relations are: \\
(1) $[P_1,\alpha,P_2]+[P'_1,\alpha',P'_2]=
     [P_1\oplus P'_1,\alpha\oplus\alpha',P_2\oplus P'_2]$, \\
(2) $[P_1,\alpha,P_2]+[P_2,\beta,P_3]=[P_1,\beta\alpha,P_3]$, \\
(3) $[P_1,\alpha,P_2]=0$ if $\alpha=f^*(\alpha_0)$
for some $\alpha_0:P_1\cong P_2$.
\\
By (1), every element of $K_0(f)$ has the form $[P,\alpha,A^n]$.

Bass showed \cite[VII.5.3]{b} that there is an exact sequence
for each $f:A\rightarrow B$:
\begin{equation}\label{seq:K0}
K_1(A) \map{f^*} K_1(B) \map{\partial} K_0(f) \map{} K_0(A) \map{} K_0(B),
\end{equation}
where for $g\in GL_n(B)$ we have $\partial([g])=[A^n,g,A^n]$.
Since we do not know if the corresponding sequence is exact for a 
quasi-projective map $f:X\to S$, we will restrict to the affine case
in this section and the next.
\goodbreak

\begin{lemma}[Excision]\label{excision} 
Let $f:A\to B$ be a ring homomorphism, and $I$ is an ideal of $A$
mapping isomorphically onto an ideal of $B$; write
$\bar{f}:A/I\subset B/I$ for the induced map.
Then excision holds for $K_n$ for all $n\le0$: 
$K_n(f) \cong K_n(\bar{f})$.
\end{lemma}

\begin{proof}
It suffices to consider the case $n=0$.
Because $K_0(A,I)\cong K_0(B,I)$ \cite[Ex.\,II.2.3]{WK} and
$K_1(A,I)\to K_1(B,I)$ is onto \cite[III.2.2.1]{WK},
the double-relative group vanishes: $K_0(A,B,I)=0$. Applying
contraction, we also have $K_{-1}(A,B,I)=0$. The result 
now follows from the exact sequence
\[
K_0(A,B,I) \to K_0(f) \to K_0(\bar{f}) \to K_{-1}(A,B,I).
\qedhere\]
\end{proof}

\begin{remm}
The failure of Lemma \ref{excision} in the non-affine setting was
investigated in \cite[A.5--6]{PW}. For example, if $X$ is the normalization
of $S$ and the support $Y$ of the conductor $\cc$ is 1-dimensional, 
the obstruction is $K_0(S,X,Y)\cong H^1(Y,\cc/\cc^2\oo\Omega_{X/S})$.
\end{remm}

As observed by Bass and Murthy long ago \cite{bm}, 
the determinant $K_0(S)\to\Pic(S)$ induces a surjective homomorphism
\begin{equation}\label{eq:det}
\det:K_0(f)\to \Pic(f), \quad 
\det[P_1,\alpha,P_2] = [\det(P_1),\det(\alpha),\det(P_2)].
\end{equation}
%By Lemma \ref{Pic=I}, 
%we may interpret the (relative) determinant as a map $K_0(f)\to\cI(f)$.
Since $SK_0(S)$ is the kernel of $\det\!:K_0(S)\to\Pic(S)$, we write
$SK_0(f)$ for the kernel of $\det:K_0(f)\to\Pic(f)$.

Recall \cite[II.4.2]{WK} that a $\lambda$-ring $K=\Z\oplus\widetilde{K}$
has a {\it positive structure} if it contains a $\lambda$-semiring $P$
(positive elements) including $\mathbb N$, such that every element of
$\widetilde{K}$ can be written as a difference of positive elements, the
augmentation $\epsilon:K\to\Z$ sends $P$ to $\N$ and, if $p\in P$ has
$\epsilon(p)=n$, then $\lambda^ip=0$ for $i>n$ and $\lambda^np$ is a unit.
The {\it line elements} are $\{p\in P:\epsilon(p)=1\}$; they form
a subgroup of the units of $K$.

\begin{proposition}\label{lambda-ops}
Let $f: A\to B$ be a homomorphism of commutative rings.
The operations $\lambda^i[P_1,\alpha,P_2]=
[\Lambda^iP_1,\Lambda^i\alpha,\Lambda^iP_2]$
give $\Z\oplus K_0(f)$ the structure of a $\lambda$-ring
with a positive structure. The top two ideals in the $\gamma$-filtration
are $F^1_\gamma=\widetilde{K}_0$ and $F^2_\gamma=SK_0(f)$, and the
group of its line elements is $\Pic(f) \cong F^1_\gamma/F^2_\gamma$.
\end{proposition}

\begin{proof}
Given $f:A\to B$, choose a surjection $\pi:\Z[X]\to B$
from  a polynomial ring $\Z[X]$ in many variables to $B$; let $R$ be
the pullback ring $R=\{(p,a)\in\Z[X]\times A:\pi(p)=f(a)\}$,
with $\tilde{f}: R\to\Z[X]$ the projection.
Since $K_1(\Z[X])=\pm1$ and $K_0(\Z[X])=\Z$, 
we have $K_0(\tilde{f})\map{\cong} \widetilde{K}_0(R)$, and this map is
compatible with the operations $\lambda^i$.
Similarly, we have $\Pic(\tilde{f})\cong\Pic(R)$. 
By Excision \ref{excision} for $K_0$ and $\Pic$, %(Lemma \ref{excision}), 
$K_0(\tilde{f})\cong K_0(f)$ and $\Pic(\tilde{f})\cong\Pic(f)$.  
Hence $\Z\oplus K_0(f)\cong\Z\oplus\widetilde{K}_0(R)$ is a $\lambda$-ring.
%Thus the long exact sequences reduce to
%$K_0(f) \map{\cong} \tilde{K}_0(R)$ and
%$\Pic(f)\map{\cong} \Pic(R).$
Thus the result follows from the fact that the operations $\lambda^i$
make $K_0(R)$ into a $\lambda$-ring, with $F^2_\gamma=SK_0(R)$, and
$\widetilde{K}_0(R)/SK_0(R)\cong\Pic(R)$.
\end{proof}

%Here is a vanishing statement for subintegral extensions.
Recall (Swan \cite{swan}) that an extension $A\subset B$ is 
said to be {\it subintegral} if 
$B$ is integral over $A$, and $\Spec(B)\to\Spec(A)$ is a bijection 
inducing isomorphisms on all residue fields.

\begin{proposition}\label{subint} (Ischebeck)
If $f:A\hookrightarrow B$ is subintegral then $K_0(f)\cong \Pic(f)$,
$K_n(f)=0$ for all $n<0$,
and there is an exact sequence
\[
1 \to B^\times/A^\times \to K_0(f) \to K_0(A) \to K_0(B) \to 0.
\]
%$K_0(f)\cong B^\times/A^\times$, and
\end{proposition}
  
\begin{proof}
When $A\subset B$ is subintegral, Ischebeck proved in \cite[Prop.\,7]{is}
that the natural map $K_{0}(A)\to K_{0}(B)$ is surjective and 
%Ischebeck also proved in {\it loc.\,cit.} that 
$SK_1(A)\to SK_1(B)$ is onto, so the cokernel
of $K_1(A)\to K_1(B)$ is $B^\times/A^\times$. The exact sequence follows
from \eqref{seq:K0}.  Finally, Ischebeck proved in 
\cite[p.\,331]{is} that the determinant \eqref{eq:det} induces an 
isomorphism from the kernel of $K_{0}(A)\to K_{0}(B)$ onto the kernel of
$\Pic(A)\to\Pic(B)$. The result now follows from \eqref{seq:K0}.

Replacing $A$ and $B$ by Laurent polynomial extensions,
the Fundamental Theorem of $K$-theory \cite[III.4.1]{WK} implies that 
$LK_{n}(f)\cong K_{n-1}(f)$ and
$K_{-1}(f)\cong L\Pic(f)$.  Since $A[t,1/t]\subset B[t,1/t]$ is
subintegral, we have $L\Pic(f)=0$ by Proposition 5.6 of \cite{SW}.
This shows that that $K_n(f)=0$ for all $n<0$.
\end{proof}

Given an extension $f:A\hookrightarrow B$, let
$i:A\hookrightarrow{}^+\!A$ be the seminormalization of $A$ in $B$
and ${}^+f:{}^+\!A\hookrightarrow B$ the induced map. 
There is an exact sequence
\[
\cdots\to K_n(i) \to K_n(f) \to K_n({}^+f) \to K_{n-1}(i)\to\cdots.
\]

\begin{corollary}
%Let ${}^+\!A$ be the seminormalization of $A$ in $B$. Then the
$K_n(f) \smap{\cong} K_n({}^+\!f)$ for $n<0$, and
the following sequence is exact.
\[
0 \to K_0(i) \to K_0(f) \to K_0({}^+\!f) \to 0.
%0 \to K_0(A,{}^+\!A) \to K_0(A,B) \to K_0({}^+\!A,B)\to 0
\]
\end{corollary}

%Let $i$ denote the inclusion $A\subset{}^+\!A$.
\begin{proof}
By Proposition \ref{subint} and \cite[Lemma 3.3]{SW}, the map
$K_0(i)\cong\Pic(i) \to \Pic(f)$ is an injection. Since it factors
through $K_0(i) \to K_0(f)$, the latter map is an injection.
Since $K_{n}(i)=0$ for $n<0$, again by Proposition \ref{subint}, we are done. 
\end{proof}

\goodbreak
\section{The $W(A)$-module structure on $NK_0(f)$ and $N\Pic(f)$}

In this section, we fix a ring homomorphism $f:A\to B$
and show that $NK_0(f)$ and $N\Pic(f)$ are continuous
modules over the ring $W(A)$ of big Witt vectors, so that
\begin{equation}\label{seq:NK0}
NK_1(A) \to NK_1(B) \map{\partial} NK_0(f) \to NK_0(A) \to NK_0(B)
\end{equation}
is a sequence of $W(A)$-modules. Recall that $(1+tA[[t]])^\times$ is
the underlying abelian group of the ring $W(A)$;
a $W(A)$-module is continuous if every element is killed
by one of these ideals $(1+t^nA[[t]])^\times$.

\smallskip
We first recall the continuous $W(R)$-module structure on $NK_*(A)$
when $R$ is commutative and $A$ is an $R$-algebra, 
due to Stienstra \cite{Jan}. As $NK_*(A)$  is a continuous module, 
it suffices to describe multiplication by $(1-rt^m)$, $r\in R$. 
Setting $S=R[s]/(s^{m}-r)$, the inclusion $i:R\subset S$ induces a 
base change functor $i^*:\bP(A[t])\to\bP(A\oo_RS[t])$ and a transfer map 
$i_*:\bP(A\oo_RS[t])\to\bP(A[t])$. If $\sigma$ denotes the $S$-algebra map 
$S[t]\to S[t]$, $\sigma(t)=st$, then the composition 
$F=i_*\sigma^{*}i^*$ is an additive self-functor of $\bP(A[t])$. 
As noted in \cite[1.5]{Wmod}, the composition 
$\bP(A)\to\bP(A[t])\map{F}\bP(A[t])\to\bP(A)$ is $\oo_RS$,
so $F$ induces multiplication by $m$ on the summand $K_*(A)$ of $K_*(A[t])$;
the restriction of $F$ to $NK_*(A)$ is multiplication
by $(1-rt^{m})*$. If $A\to B$ is an $R$-algebra map,
$NK_*(A)\to NK_*(B)$ is a homomorphism of continuous $W(R)$-modules.

We can adapt these formulas to define a multiplication by 
$(1-at^{m})*$ on $K_0(f)$ and $NK_0(f)$ when $a\in A$: send
$[P_1,\alpha,P_2]$ to $[F(P_1),F(\alpha),F(P_2)]$. It is clear 
from \eqref{seq:K0} that $(1-at^{m})*$ is compatible with the exact 
sequence \eqref{seq:NK0}. A priori, though, the maps $(1-at^{m})*$ 
do not fit together to make $NK_{0}(f)$ into a $W(A)$-module.
%
%The inclusion $i:A\subset S$ induces functors
% $i^*: \bP(f[t])\to \bP(f_S[t])$ and
%$i_*: \bP(f_S[t])\to \bP(f[t])$, and $\sigma_*$ induces a self-map
%of $\bP(f_S[t])$.  The composition $F=i_*\sigma_{*}i^*$
%\[
%\bP(f[t])\map{i^*} \bP( f_S[t])\map{\sigma_{*}} 
%\bP(f_S[t])\map{i_*} \bP(f[t])
%\]
%is an additive functor $\bP(f[t])\to \bP(f[t])$, taking
%$(P_1,\alpha,P_2)$ to $(F(P_1),F(\alpha),F(P_2))$, 
%This formula shows that 
%the induced map $K_{0}F: K_{0}(f[t])\to K_{0}(f[t])$ respects the
%direct sum decomposition $K_{0}(f[t])\cong K_{0}(f)\oplus NK_{0}(f)$. 
%By Lemma 1.5 of \cite{Wmod}, it is multiplication by $m$ on $K_0(f)$,
%and we define the restriction of $F$ to $NK_{0}(f)$ to be the 
%multiplication map $(1-at^{m})*$. 

\begin{proposition}\label{ctn-W}
For any homomorphism $f:A\to B$,
$NK_0(f)$ is a continuous $W(A)$-module, and \eqref{seq:NK0}
is an exact sequence of continuous $W(A)$-modules.
\end{proposition}

\begin{proof}
As in the proof of Proposition \ref{lambda-ops}, write $B=\Z[X]/I$,
where $\Z[X]$ is a polynomial ring. Let $R$ denote the pullback ring
$A\times_B\Z[X]$, and write $\tilde{f}:R\to\Z[X]$ for the quotient map.
Since $NK_*(\Z[X])=0$, we have 
$NK_n(\tilde{f})\cong NK_n(R)$ for all $n$.
Since $A=R/I$, Lemma \ref{excision} and \cite{Wmod1} imply that the 
groups $NK_0(f)\cong NK_0(\tilde{f})\cong NK_0(R)$ are continuous 
$W(R)$-modules.  

Since $W(A)=W(R)/W(I)$, where $W(I)=1+tI[[t]]$,
we are reduced to showing that $(1-rt^m)$ acts as zero on $K_0(f)$ 
whenever $r\in I$. When $r$ is in the kernel $I$ of $R\to A$, 
the ring $A\oo_RS$ is just $A[s]/(s^m)$, so $(1-rt^m)$ and $(1-0t^m)$
act identically on $K_0(f[t])$. This shows that $(1-rt^m)$ 
acts as zero on $K_0(f)$ and proves that the action of
$W(A)$ on $K_0(f)$ is well defined and continuous.
\end{proof}

\medskip

Applying $N$ to the determinant described in \eqref{eq:det},
%exact sequence $0\to SK_0(f)\to K_0(f)\map{det}\cI(f)\to0$ 
we get an exact sequence
\[
0\to NSK_0(f)\to NK_0(f)\map{\det} N\Pic(f)\to 0.
\]
If $[P,\alpha,A[t]^n]$ is in $NK_0(f)$ then
$\det[P,\alpha,A[t]^n]=[\det(P),\det(\alpha),A[t]]$.

\begin{theorem}\label{NI}
For any homomorphism $f:A\to B$,
% Let $f: A\hookrightarrow B$ be a ring extension. Then 
$N\Pic(f)$ is a continuous $W(A)$-module, and
$\det: NK_0(f)\to N\Pic(f)$ is a $W(A)$-module homomorphism.
\end{theorem}

\begin{proof}
Since the group $NK_{0}(f)$ is a continuous $W(A)$-module by 
Proposition \ref{ctn-W}, it is enough to show that $NSK_{0}(f)$ is 
closed under multiplication by $W(A)$. Since every
element of $W(A)$ can be written as $\prod_{m>0}(1-a_{m}t^{m}),$
with $a_{m}\in A$, and for any element $u$ of $NK_0(f)$ there is an
$n$ so that $\prod_{m\ge n}(1-a_{m}t^{m})*u=0$, it is enough to show that 
$NSK_{0}(f)$ is closed under multiplication by $(1-at^{m})$ 
for any $a\in A$ and $m\ge1$.

It is enough to show that $F=i_*\sigma^*i^*$ sends $SK_{0}(f[t])$ to itself.
We now modify the argument of \cite[4.1]{DW}. 
Fix $u=[P,\alpha,A[t]^n]$ in $SK_0(f[t])$; 
By Remark \ref{[L,a]}, %by Lemma \ref{Pic=I}, 
$\det(u)=0$ implies that $\det(P)=A[t]$ and $\det(\alpha)\in A$. 
By naturality of $\det$,
$\sigma^*i^*(u)=[P\oo S,\alpha\oo S, S[t]^n]$, $\det(P\oo S)=S[t]$,
$\det(\alpha\oo S)\in S$ and
$F(u)=[i_*(P\oo S),i_*(\alpha\oo S),A[t]^n]$.
By Corollary 3.2 of \cite{DW} applied to $A[t]\subset S[t]$,
$\det(i_*(P\oo S))=A[t]$ and $\det(\alpha\oo S)=\det(\alpha)^m\in A$, so 
$\det(F(u))=0$.
\end{proof}

\begin{corollary}
If ${\rm char}(A)=p$ then $N\Pic(f)$ is a $p$-group. \\
If ${\mathbb Q}\subseteq A$ then $N\Pic(f)$ is an $A$-module.
\end{corollary}

\begin{proof}
Any continuous $W(A)$-module has these properties; see
\cite[3.3]{Wmod1}.
\end{proof}

%\bigskip\goodbreak
\section{Sheaf properties of $N\Pic(f)$}

When $f: X\to S$ is a faithful affine morphism of schemes,
let $\cI(f)_\zar$ denote the Zariski sheaf
$f_*\cO_{X}^{\times}/\cO_{S}^{\times}$ on the category $Sm/S$
of smooth schemes over $S$; by \cite[4.4]{SW},
$\cI(f)_\zar$ is also an \'etale sheaf, and $H_\et^{0}(S,\cI(f)_\zar)=
H_\nis^{0}(S,\cI(f)_\zar)=\Pic(f)$. 
Our choice of $Sm/S$ is dictated by the need to not only include \'etale
extensions but be closed under product with $\A^1_S\smap{\pi} S$. %$\A^1$.

Let $\pi^*\cI(f)$ denote the restriction of $\cI(f)_\zar$ 
to $Sm/\A^1_S$ along $\pi$. % $\A^1_S\smap{\pi} S$.  
Its direct image $\pi_*(\pi^*\cI(f))$ is 
the Zariski sheaf $\cI(f)_\zar\oplus \cNI(f)$ on $Sm/S$,
where $\cNI(f)$ denotes the Zariski sheaf on $Sm/S$ associated to the presheaf 
$U\!\mapsto\!N\Pic(f\!\times_S\!U).$ %on $S$. 

\begin{theorem}\label{H^0}
Let $f: X\to S$ be a faithful affine morphism of schemes. Then
$\cNI(f)$ is an \'etale sheaf on $S$. Moreover,
\[
H_\et^{0}(S, \cNI(f))= H_\zar^{0}(S,\cNI(f))=N\Pic(f).
\]
\end{theorem}

\begin{proof}
Since $\pi^*\cI(f)$ is an \'etale sheaf on $\A^1_S$, its direct image
$\pi_*\pi^*\cI(f)$ is an \'etale sheaf on $S$; since
$\pi_*\pi^*\cI(f) \cong\cI(f)_\zar\oplus\cNI(f)$, 
$\cNI(f)$ is also an \'etale sheaf. Since
\[
H^0_\et(S,\pi_*\pi^*\cI(f)) = H^0_\et(\A^1_S,\pi^*\cI(f)) = \Pic(f[t])
= \Pic(f)\oplus N\Pic(f),
\]
we see that $H_\et^{0}(S, \cNI(f))=N\Pic(f)$.
If $S_s$ is a Zariski local scheme of $S$, this shows that the stalk 
$\cNI(f)_s=H^0_\zar(S_s,\cNI(f))$ equals $H^0_\et(S_s,\cNI(f))$.
\end{proof}

\begin{example}\label{snormal}
If $f$ is seminormal, the sheaf $\cNI(f)$ vanishes and $N\Pic(f)=0$.
This follows from Theorem \ref{H^0} and \cite[1.5]{ss}, 
which states that $N\Pic(A,B)=0$ when $A$ is seminormal in $B$.
\end{example}

We now modify an argument of Vorst \cite{Vorst} and van der Kallen \cite{vdk}.
Suppose that $\Spec(A)=\bigcup_{i=0}^{r} U_i$, where $U_i=\Spec(A_{s_{i}})$.
Given a presheaf $F$ of abelian groups on $\Spec(A)$,
% from commutative rings to abelian groups,
we write $C^\mathdot(\{U_i\},F)$ for the augmented \v Cech complex:
\[
0\to F(A)\map{\epsilon} \prod_{i=0}^r F(A_{s_{i}})\to 
\prod_{0\le i< j\leq r} F(A_{s_{i}s_{j}})\to \dots 
\to F(A_{s_{0}s_{1}\cdots s_r})\to 0.
\]
Given $s\in A$, we have an $A$-algebra map $\sigma:A[x]\to A[x]$
determined by $\sigma(x)=sx$. We
write $NF(A)_{[s]}$ for the direct limit of 
$F(A[x])\smap{\sigma}F(A[x])\smap{\sigma}\cdots$.
Suppose that for all $0\leq i_{0}<i_{1}<\dots<i_{p}\leq r$ and $j\leq p$: 
\begin{equation}\label{NPIC}
NF(A_{s_{i_0}\cdots s_{i_j}\cdots s_{i_p}}[x]) \cong 
NF(A_{s_{i_0}\cdots \hat{s}_{i_j}\cdots s_{i_p}}[x])_{[s_{i_j}]}.
\end{equation}
In this situation, Vorst proved \cite[1.2]{Vorst} that
the sequence $C^\mathdot(\{U_i\},NF)$ is always exact.
He also proved that $F=NK_n$ satisfied \eqref{NPIC}, so that
$C^\mathdot(\{U_i\},NK_n)$ is exact for all $n$.
(See \cite[1.4]{Vorst} or \cite[V.8.5]{WK}; 
the nonzerodivisor hypothesis is unnecessary by \cite{TT}.)

\begin{remark}\label{cech-NU}
It is easy to see (and follows from Vorst's result \cite[1.2]{Vorst}) that
the functor $NU(A)=(A[t])^\times/A^\times$ satisfies \eqref{NPIC}. From the
exact sequence of complexes
\[
0 \to C^\mathdot(\{U_i\},NU) \to C^\mathdot(\{U_i\},NU(-\otimes_AB)) \to
C^\mathdot(\{U_i\}, NU(-\otimes_AB)/NU) \to 0
\]
we see that $C^\mathdot(\{U_i\},F)$ is also exact
for the functor $F(A_s)=NU(B_s)/NU(A_s)$.
\end{remark}

\begin{lemma}\label{cech-A}
$C^\mathdot(\{U_i\},N\Pic)$ 
%and $C^\mathdot(\{U_i\},N\cI(f))$ are 
is always an exact sequence.
\end{lemma}

\begin{proof}
By Theorem 4.2 of \cite{wei}, given $s\in A$ we have 
$N\Pic(A_{s})\cong N\Pic(A)_{[s]}$ and hence
$N\Pic(A_{s}[x])\cong N\Pic(A[x])_{[s]}$. This implies that 
$N\Pic$ satisfies \eqref{NPIC}.  Vorst's result shows that
$C^\mathdot(\{U_i\},N\Pic)$ is an exact sequence.
\end{proof}

We apply these considerations to the presheaf 
$N\Pic(f):U\mapsto N\Pic(f|_U)$.

\begin{lemma}\label{cech}
Suppose that $\Spec(A)=\cup_{i=0}^{n} U_i$, where $U_i=\Spec(A_{s_{i}})$.
If $f:A\hookrightarrow B$ is a ring extension, 
%so $U\mapsto N\cI(f|_U)$ is a presheaf, 
the complex $C^\mathdot(\{U_i\},N\Pic(f))$ is exact.
%  \footnotesize
\[
0\to N\Pic(A, B)\to \prod_{i=0}^{n} N\Pic(A_{s_i}, B_{s_i})\to 
\prod_{i_i<i_2} N\Pic(A_{s_{i_1}s_{i_2}}, B_{s_{i_1}s_{i_2}})\to \cdots 
%\to N\cI(A_{s_{i_1}...s_{i_n}},B_{s_{i_1}...s_{i_n}})\to0
\]
\end{lemma}

\begin{proof}
Let $^{^+}\!\!A$ denote the subintegral closure of $A$ in $B$, so 
$^{^+}\!\!A$ is seminormal in $B$ and we have
$A\subset {}^{^+}\!\!A \subset B$.  By \cite[Prop.\,4.1]{vs},
we have an exact sequence
\[
1\to N\Pic(A,^{^+}\!\!\!A ) \to N\Pic(A, B) \to N\Pic(^{^+}\!\!A,B) \to 1.
\] 
By Example \ref{snormal}, %Theorem 1.5 \cite{ss}, $N\Pic(^{^+}\!\!A,B)=0$.
the third term vanishes and we have
$N\Pic(A,^{^+}\!\!\!A )\cong N\Pic(A,B)$. Thus we may assume that $B$
is subintegral over $A$. In this case, Ischebeck proved
\cite[Prop.\,7]{is} that $N\Pic(A)\to N\Pic(B)$ is surjective. Now the 
result follows from Remark \ref{cech-NU}, Lemma \ref{cech-A} and the
long exact cohomology sequences associated to 
\begin{align*}
0\to C^\mathdot(\{U_{i}\},F)\to C^\mathdot(\{U_{i}\},N\Pic(f))&
\to C^\mathdot(\{U_{i}\},N\Pic(f)/F)\to0, \\
0\to C^\mathdot(\{U_{i}\},N\Pic(f)/F)\to C^\mathdot(\{U_{i}\},N\Pic)&
\to C^\mathdot(\{f^{-1}(U_{i})\},N\Pic)\to 0.
\qedhere\end{align*}
\end{proof}

\begin{theorem}\label{quasi}
 Let $f:A\hookrightarrow B$ be an extension of rings. 
Then:
% \begin{enumerate}\item 
$$H_\et^{q}(\Spec(A), \cNI)=\begin{cases}
  N\Pic(f) & {\rm if}~ q=0\\   0 & {\rm if}~ q> 0 \end{cases}$$
%\item $H_\et^{q}(\Spec(A), \pi_{*}^{t}\cI^{t})= \begin{cases}
%  \cI(f)\oplus N\cI(f) & {\rm if}~ q=0\\
%  H_\et^{q}(\Spec(A), \cI) & {\rm if}~ q>0      \end{cases}$
%
%\item  $H_\et^{q}(\Spec(A), \pi_{*}^{T}\cI^{T})= \begin{cases}
%  \cI(f)\oplus N\cI(f) \oplus N\cI(f)\oplus L\cI(f) & {\rm if}~ q=0\\
%  H_\et^{q}(\Spec(A),\cI)\oplus
%  H_\et^{q}(\Spec(A), \mathcal{LI})&{\rm if}~q>0
%   \end{cases} $
%\item If $f$ is finite and connected then 
%$H_\nis^{q}(\Spec(A), \pi_{*}^{T}\cI^{T})= H_\nis^{q}(\Spec(A), \cI)$
% for $q> 0$.
%\end{enumerate}
\end{theorem}

\begin{proof}
The case $q=0$ is given by Theorem \ref{H^0}.
By Lemma \ref{cech}, the $\rm{\check{C}ech}$ cohomology groups
$\check{\mathrm{H}}^q(\Spec(A),\cNI)$ vanish for $q>0$. 
Using the Cartan criterion \cite[III.2.17]{Milne}, 
$H_\et^{q}(\Spec(A),\cNI)$ equals 
$\check{\mathrm{H}}^q(\Spec(A),\cNI)=0$ for $q>0$.
\end{proof}

% (2) This follows from (1) and the proof of Theorem \ref{H^0}, because
%  cohomology commutes with finite direct sums.
% (3) Since $\cI$ is a contracted functor, $\pi_{*}^{T}\mathcal{F}^{T}
%\cong \cI\oplus \cNI \oplus \cNI \oplus \mathcal {LI}$ as a
%  \'etale (or Nisnevich) sheaf on $S$ by 
%\cite[0.1]{SW}.  %Theorem \ref{contracted}. 
%Now  the result follows from (1).
% 
% (4) Since $f$ is connected and finite, the nisnevich sheaf 
%$f_{*}\Z_{X}/\Z_{S}=0$. Then by (3), we get the result.

\begin{corollary}\label{et=zar}
 Let $f: X\to S$ be a faithful affine morphism of schemes. Then 
\[
H_\et^{*}(S, \cNI)\cong H_\zar^{*}(S, \mathcal{NI)}.
\]
\end{corollary}

\begin{proof}
Consider the site change map $\tau: S_\et\to S_\zar$. Then by
Theorem \ref{quasi}, the higher direct image sheaves
$R^{q}\tau_{*}\,\cNI$ vanish for $q> 0$. Therefore the Leray spectral
sequence degenerates, yielding the result.
\end{proof}

\begin{remm}
More generally, if $f:X\to S$ is any morphism of schemes then $\cO_S^\times$
may not inject into $f_*\cO_X^\times$. In this case, if we interpret
$f_*\cO_X^\times/\cO_S^\times$ as the mapping cone of 
$\cO_S^\times\to f_*\cO_X^\times$ (a complex of Zariski sheaves)
and use sheaf hypercohomology, then Theorem \ref{H^0} remains valid.
However, Theorem \ref{quasi} may fail in this setting.
\end{remm}

%Now suppose that $f$ is seminormal.
% Since seminormality is local on the base and $f$ is affine, we may
% assume that $X= \Spec(B)$ and $S=\Spec(A)$. By Theorem 1.5 of
% \cite{ss}, $N\cI(f)=0$ if and only if $A$ is seminormal in $B$. Then
% it is easy to check that $\cNI^\zar=0$ if and only if $f$ is
% seminormal. 
%Now the result follows from Corollary \ref{et=zar}.

\section{Module Structures on $NK_n(f)$}\label{sec:module}

Given an exact functor $F:\mathcal P\to\mathcal Q$, the 
relative $K$-theory groups $K_n(F)$ fit into an exact sequence
\[
\cdots \map{F} K_{n+1}\mathcal Q \map{\partial} K_n(F) \to
K_n\mathcal P \map{F} K_n\mathcal Q \map{\partial} \cdots
\]
ending in $K_0\mathcal Q \map{\partial} K_{-1}(F)$.
Waldhausen showed that the $K_n(F)$ are the homotopy groups 
$\pi_{n+2}|wS_\mathdot(S_\mathdot F)|$ ($n\ge0$), 
where $S_nF$ denotes the category of pairs 
$$
(P_*,Q_*) = 
(P_1\cof P_2\cof\cdots\cof P_n, Q_0\cof Q_1\cof\cdots\cof Q_n)
$$
($P_i\in\mathcal P$ and $Q_j\in\mathcal Q$), together with choices of 
$Q_i/Q_j$ for $i>j$, such that $F(P_*)$ is 
$Q_1/Q_0\cof\cdots\cof Q_n/Q_0$.
(See \cite[1.5.4--7]{Wald} or \cite[IV.8.5.3]{WK}.)
 
\begin{example}\label{ex:f[t,1/t]}
If $A$ is a ring, we write $\bP(A)$ for the category of 
finitely generated projective $A$-modules. Given a ring homomorphism
$f:A\to B$, we have an exact functor $\bP(f): \bP(A)\to\bP(B)$;
by abuse, we write $K_*(f)$ for $K_*\bP(f)$.  Writing $f[t]$ for
$A[t]\to B[t]$, we have $K_*(f[t])=K_*(f)\oplus NK_*(f)$.
The Fundamental Theorem of $K$-theory easily extends to 
the relative setting, yielding 
\[
K_*(f[t,1/t])\cong K_*(f)\oplus NK_*(f)\oplus NK_*(f)\oplus K_{*-1}(f).
\]
\end{example}

Let $A$ be a commutative ring. As in \cite{WK}, we write
%$\bP(A)$ for the category of finitely generated projective $A$-modules,
$\bEnd(A)$ for the category of pairs $(P,\alpha)$, where $P$ in $\bP(A)$
and $P\smap{\alpha}P$ is an endomorphism, and write
${\bNil}(A)$ for the full subcategory of $\bEnd(A)$ consisting of
all $(P,\alpha)$ with $\alpha$ nilpotent. As pointed out in
\cite[II.7.4]{WK}, $K_*\bEnd(A)\cong K_*(A)\oplus \End_*(A)$ and
$K_*\bNil(A)\cong K_*(A)\oplus \Nil_*(A)$, where
$\End_*(A)$ is a graded-commutative ring and $\Nil_*(A)$ is a graded
$\End_*(A)$-module.
By naturality, the exact functors $\bNil(f):\bNil(A)\to\bNil(B)$
yield relative groups
$K_*\bNil(f)\cong K_*(f)\oplus \Nil_*(f)$.

The category $\bNil(A)$ is equivalent to the category 
${\bf H}_{1,t}(A[t])$ of $t$-primary torsion $A[t]$-modules $M$ 
with $pd_{A[t]}M=1$. Specifically, if $(P,\nu)$ is in $\bNil(A)$, 
and we write $P_\nu$ for the $A[t]$-module $P$ on which $t$ acts 
as $\nu$, then $P_\nu$ has
projective dimension~1 over $A[t]$.
% (see \cite[II.7.8.2]{WK}). 
The Fundamental Theorem
(\cite[V.8.2]{WK}) implies that $\Nil_n(A)\cong NK_{n+1}(A)$.
We also have $K\bP(A[t])\cong K{\bf H}(A[t])$ 
(see e.g., \cite[V.3.2]{WK}).
%From this construction, it is clear that
%By naturality, the exact functors
% $\bEnd(A)\to\bEnd(B)$ and
%$\bNil(f):\bNil(A)\to\bNil(B)$ yield relative groups 
%and ${\bf H}_{n,S}(A)\to {\bf H}_{n,S}(B)$ 
%$K_*\bEnd(f)\cong K_*(f)\oplus \End_*(f)$ and
%$K_*\bNil(f)\cong K_*(f)\oplus \Nil_*(f)$.
%$K_*{\bf H}_{n,S}(f)$,
%$K_*\bNil(f)\cong K_*(f)\oplus \Nil_*(f)$.

\begin{proposition}\label{Nil=NK}
There is a natural isomorphism $\Nil_n(f) \cong NK_{n+1}(f)$.
\end{proposition}

\begin{proof}
From the diagram of exact categories
\[\begin{CD}
\bNil(A) @>{\cong}>> {\bf H}_{1,t}(A[t]) @>>> 
{\bf H}(A[t]) @<{\cong}<< \bP(A[t]) @>>> \bP(A[t,1/t]) \\
@VVV           @.   @.        @VVV               @VVV  \\
\bNil(B) @>{\cong}>> {\bf H}_{1,t}(B[t]) @>>> 
{\bf H}(B[t]) @<{\cong}<< \bP(B[t]) @>>> \bP(B[t,1/t])
\end{CD}\]
%and the fact that $K\bP(A[t])\cong K{\bf H}(A[t])$ 
%(see e.g., \cite[V.3.2]{WK}),
we get a fibration sequence of $K$-theory spectra
\[\begin{CD}
K\bNil(A) @>>> K(A[t]) @>>> K(A[t,1/t])  \\
@VVV        @VV{f[t]^*}V    @VV{f[t,1/t]^*}V \\
K\bNil(B) @>>> K(B[t]) @>>> K(B[t,1/t]).
\end{CD}\]
Taking vertical fibers, 
we see that there is a long exact sequence
\[
K_{n+1}(f[t])\to K_{n+1}(f[t,1/t])\to K_n\bNil(f) \to 
K_n(f[t])    \to K_n(f[t,1/t])\to
\]
and (using Example \ref{ex:f[t,1/t]}) 
an isomorphism $\Nil_n(f) \cong NK_{n+1}(f)$.
\end{proof}

\begin{lemma}
For any ring homomorphism $f:A\to B$, 
%Let $f:A\to B$ be a ring map. Then 
$\Nil_*(f)$ is a graded $\End_*(A)$-module.
\end{lemma}

\begin{proof}
%Let us write $\bNil(f)$ for the functor $\bNil(A)\to\bNil(B)$. 
A typical object in the Waldhausen category $S_n\bNil(f)$ is a pair
\[
(\mu_*,\nu_*)=
((M_1,\mu_1)\cof\cdots(M_n,\mu_n),(N_0,\nu_0)\cof\cdots(N_n,\nu_n)).
\]
There is a pairing 
$\bEnd(A)\times S.\bNil(f) \to S.\bNil(f)$ 
of simplicial Waldhausen categories, sending 
$(P,\alpha)\times(\mu_*,\nu_*)$ to
\[(
(P\oo M_1,\alpha\oo\mu_1)\!\cof\!\cdots\!\cof\!(P\oo M_n,\alpha\oo\mu_n),
(P\oo N_0,\alpha\oo\nu_1)\!\cof\!\cdots\!\cof\!(P\oo N_n,\alpha\oo\nu_n)
).\]
It induces a pairing $K_*\bEnd(A)\oo K_*\bNil(f)\to K_*\bNil(f)$.
Since the tensor product $(\alpha\oo\beta)\oo\mu\cong\alpha\oo(\beta\oo\mu)$
is associative up to natural isomorphism, the two pairings 
$$
\bEnd(A)\times\bEnd(A)\times S.\bNil(f) \to S.\bNil(f)
$$
agree up to natural isomorphism, making 
$K_*\bNil(f)$ a graded $K_*\bEnd(A)$-module.
In particular, $\Nil_*(f)$ is a graded module over $\End_*(A)$.
\end{proof}

Recall that the ring $W(A)$ of big Witt vectors has underlying 
abelian group $(1+tA[[t]])^\times$\!. Almkvist's theorem \cite[II.7.4.3]{WK} 
states that $[P,\alpha]\mapsto\det(1-t\alpha)$ maps $\End_0(A)$ 
isomorphically onto the subring of $W(A)$ whose
underlying abelian group consists of all quotients $f(t)/g(t)$ of 
polynomials in $1+tA[t]$. The intersection of the ring
$\End_0(A)$ with the ideal $(1+t^mA[[t]])$ of $W(A)$
is the ideal $I_m=\{1+t^m(f/g)\}$ of $\End_0(A)$, and
%ideal $I_m=\{1+t^m(f/g)\}$ of $\End_0(A)$ is the intersection
%of $\End_0(A)$ with the ideal $(1+t^mA[[t]])$ of $W(A)$, and
$\End_0(A)/I_m\cong W(A)/(1+t^mA[[t]])$. In particular, $W(A)$ is
the completion of $\End_0(A)$ with respect to the $t$-adic filtration.

We say that an $\End_0(A)$-module $M$ is {\it continuous} if
for every $x\in M$ there is an $m$ so that $I_m\cdot x=0$.
Thus every continuous $\End_0(A)$-module $M$ is also continuous 
as a $W(A)$-module: for every $x\in M$ we have $(1+t^mA[[t]])\cdot x=0$
for some $m$.

The exact functors $F_n, V_n: \bNil(A)\to\bNil(A)$, defined 
by $F_{n}(P,\nu)=(P,\nu^{n})$ and $V_n(Q,\nu)=(Q[t]/(t^{n}-\nu),t)$,
commute with $\bNil(A)\to\bNil(B)$.  Hence they
induce exact endofunctors $F_n, V_n$ on $S.\bNil(f)$ by
$F_n(\mu_*,\nu_*)=(F_n(\mu_*),F_n(\nu_*))$ and
$V_n(\mu_*,\nu_*)=(V_n(\mu_*),V_n(\nu_*))$. 
For $a\in A$ and $n>0$, and $\nu$ in $\Nil_*(f)$, 
Almkvist's theorem associates $(1-at^n)$ to $V_n([A,a]-[A,0])$
and yields the product formula
\begin{equation}\label{W-action}
(1-at^n) * \nu = V_n([A,a]-[A,0])*\nu.
\end{equation}

Stienstra proved in \cite{Jan,Jan1} that the $\Nil_n(A)$ are continuous
$\End_0(A)$-modules, and hence $W(A)$-modules. The key step \cite[2.12]{Jan}
was showing that the projection formula holds:
\[
(V_{n}\alpha)*\nu= V_{n}(\alpha*F_{n}(\nu)) \quad\text{for}\quad
\alpha\in\End_0(A) \textrm{~and~} \nu\in\Nil_*(A).
\]
Here is the corresponding projection formula in the relative setting;
we will postpone its proof in order to get to the main result.

\begin{lemma}\label{projection}
For all $\alpha\in\End_0(A)$ and $\beta\in\Nil_{*}(f)$,
\[
(V_{n} \alpha)*\beta=V_{n}(\alpha* F_{n}(\beta)).
%(V_{n} \alpha)*(\mu_*,\nu_*)=V_{n}(\alpha*F_{n}(\mu_*, \nu_*)).
\]
%where $\alpha\in\End_0(A)$ and $(\mu_*, \nu_*)\in \Nil_{*}(f)$.
\end{lemma}

\begin{theorem}\label{NK}
 Let $f:A\to B$ be a ring map. Then the product \eqref{W-action} makes
$\Nil_{n}(f) \cong NK_{n+1}(f)$ into a continuous $W(A)$-module
for every integer $n$.
% \end{enumerate}
\end{theorem}

%Our proof will use the following modification of a result of Stienstra
%\cite[2.12]{Jan}, whose proof we postpone:

\begin{proof}
For each $m>0$, let $\bNil^m(A)$ denote the exact subcategory 
of all $(P,\nu)$ in $\bNil(A)$ such that $\nu^m=0$.
Thus we have relative groups $K_*\bNil^m(f)$ associated to
$K_*\bNil^m(A)\to K_*\bNil^m(B)$, and $K_*\bNil(f)$
is the direct limit of the $K_*\bNil^m(f)$.

Suppose that $n\ge m$.
Clearly, $F_n$ acts as zero on $\bNil^m(f)$. By the
projection formula \ref{projection}, $V_n(\alpha)$ acts as zero on the image
$\Nil^m_*(f)$ of $K_*\bNil^m(f)\smap{} K_*\bNil(f)\to\Nil_*(f)$. %when $m\ge n$.
By \eqref{W-action}, $(1- at^n)$ acts as zero on
$\Nil^m_{*}(f)$. Since $\Nil_{*}(f)$ is the union of the
$\Nil^m_*(f)$, for any $\beta\in \Nil_*(f)$ there is an $m$ such
that $(1-at^n)\cdot\beta=0$ for all $n\ge m$ and $a\in A$. This shows that
$\Nil_*(f)$ is a continuous $\End_0(A)$-module, and hence
a continuous $W(A)$-module.
\end{proof}

\begin{proof}[Proof of Lemma \ref{projection}]
Following Stienstra \cite[\S6]{Jan}, 
set $R=\Z[y_1,y_2]$, and set $\mathbf{E}=\bEnd(R;S_6)$, where 
$S_6$ is the multiplicative subset of $R[x]$ generated by $x$ 
and $x^n-y_1^ny_2$. %as in \cite[\S6]{Jan}. 
As pointed out in {\it loc.\,cit.}, there is a multi-exact pairing
\[
\Theta: \mathbf{E} \times \bEnd(A)\times\bNil(B) \to \bNil(B)
\]
sending $(E,\omega)$, $(P,\alpha)$ and $(N,\nu)$ to 
$(E\oo_R(P\oo_AN),\omega\oo1)$, where $P\oo_AN$ is ragarded as an
$R$-module by letting $y_1$ and $y_2$ act as $\alpha\oo1$ and $1\oo\nu$.
As this pairing is natural in $B$, we may replace $\bNil(B)$ by $S.\bNil(f)$.
This yields (among other things) a product
\[
\Theta_*: K_0\mathbf{E} \oo \End_0(A)\oo \Nil_*(f) \to \Nil_*(f).
\]
Stienstra proves in {\it loc.\,cit.}\ %\cite[\S6]{Jan}
that the elements $[R^n,\omega]$ and $[R^n,\omega']$ agree in 
$K_0 \mathbf{E}$, where
\[
\omega = \begin{pmatrix}0 &&& y_1^ny_2\\1 &&&0\\&\ddots&&\vdots\\0&&1&0
\end{pmatrix} \quad\textrm{and}\quad
\omega' = \begin{pmatrix}0 &&& y_1y_2\\y_1&&&0\\&\ddots&&\vdots\\0&&y_1&0
\end{pmatrix}.\]
Therefore the two maps 
\[\Theta_*([R^n,\omega],-),\Theta_*([R^n,\omega'],-):
%He also proves that the resulting maps
\End_0(A)\oo\Nil_*(f) \to \Nil_*(f)
\] 
agree. Stienstra also observes that these maps send 
$[P,\alpha]\oo\beta$ to $V_n(\alpha*F_n\beta)$ and $(V_n\alpha)*\beta$,
respectively; see also \cite[p.14]{Jan1}. The projection formula follows.
\end{proof}

%\[\begin{CD}
%\bEnd(A)\times S_.\bNil(A) @>{1\times F_m}>> \bEnd(A)\times S_.\bNil(A)
%@>{\otimes}>> \bNil(A) \\
%@V{V_m}VV           @V{}VV         @VV{V_m}V \\
%\bEnd(A)\times S_.\bNil(A) @>{V_m\times 1}>> \bEnd(A)\times S_.\bNil(A)
%@>{\otimes}>> \bNil(A).
%\end{CD}\]

\section{Negative relative K-theory}

Let $f:X\to S$ be a morphism of schemes. Then we have a long exact
sequence of negative K-groups, part of which is:
\begin{equation}\label{kgroup}
\cdots\to K_{-d}(f)\to K_{-d}(S)\to K_{-d}(X)\to K_{-d-1}(f)\to K_{-d-1}(S)\to\cdots.
\end{equation}

%We say that a morphism $f:X\to S$ is $F$-regular if 
%$F(f)\to F(f\times \mathbb{A}^{r})$ is an isomorphism for all $r\geq 0$.

\begin{theorem}\label{vanishing}
Let $f\!:X\!\to\!S$ be a morphism of $d$-dimensional schemes,
essentially of finite type over a field $k$ of characteristic $0$.
Then for all $r>0$:
\begin{enumerate}
\item $K_{n}(f)=K_n(f\times\A^r)=0$ for $n\le-d-2$.
\item $K_{-d-1}(f)\cong K_{-d-1}(f\times\A^r)$
%\item If $n<-d$ then $K_n(f)\cong K_n(f\times\A^r)$ for all $r>0$.
(``$f$ is $K_{-d-1}$-regular.'') %for all $n<-d$.)

\item If $f$ is a finite map then 
$K_{-d-1}(f)\cong H_\cdh^{d}(S, f_{*}\Z/\Z)$.
\end{enumerate}
\end{theorem}

\begin{proof}
By Corollary 5.9 and Theorem 6.2 of \cite{CHSW}, %$S$ is $K_{-d}$-regular,
$K_n(S)\cong K_n(S\times\A^r)$ for all $n\le -d$,
$K_{n}(S)=0$ for $n<-d$ and $K_{-d}(S)\cong H_\cdh^{d}(S,\Z)$; 
the analogous assertions hold for $X$.
The exact sequence \eqref{kgroup} for $S$ and $S\times\A^r$
implies the first two assertions.
%The second assertion follows easily from a comparison of \eqref{kgroup}
%for $S$ with \eqref{kgroup} for $S\times\A^n$.
For (3), we have a distinguished triangle cdh sheaves on $S$,
\[
\Z\to f_{*}\Z\to f_{*}\Z/\Z\to \Z[1].
\]
Since the cdh-cohomological dimension of $S$ is at most $d$, 
$H_\cdh^{d+1}(S,\Z)=0.$ Thus the long exact sequence on 
cdh-cohomology ends in
\begin{equation*}
\to H_\cdh^{d}(S, \Z)\to H_\cdh^{d}(S, f_{*}\Z)\to 
H_\cdh^{d}(S, f_{*}\Z/\Z)\to 0.
%H_\cdh^{d+1}(S, \Z)\to H_\cdh^{d+1}(S, f_{*}\Z)\to0.
\end{equation*}
Since $f$ is finite, we have 
$H_\cdh^*(S,f_*\Z)\map{\cong}H_\cdh^*(X,\Z)$;
assertion (3) follows.
\end{proof}
%We have a spectral sequence 
%$H_\cdh^{p}(S, R^{q}f_{*}\Z)\to H_\cdh^{p+q}(X, \Z)$ and 
%$R^{q}f_{*}\Z=0$ for $q>0$ because $f$ is finite. Then 
%$H_\cdh^{r}(S, f_{*}\Z)= H_\cdh^{r}(X, \Z)$. In particular, 
%$H_\cdh^{d}(S, f_{*}\Z)= H_\cdh^{d}(X, \Z)$. 
%Now from exact sequence 
%\eqref{kgroup} and Theorem 6.2 of \cite{CHSW}, we get 
%   \begin{align*}
%     H_\cdh^{d}(S, f_{*}\Z/\Z)&=  \coker[H_\cdh^{d}(S, \Z)\to
% H_\cdh^{d}(S, f_{*}\Z)]\\
%     &=\coker[H_\cdh^{d}(S, \Z)\to H_\cdh^{d}(X, \Z)]\\
%     &=\coker[K_{-d}(S)\to K_{-d}(X)]\\
%     &= K_{-d-1}(f) \qedhere
%   \end{align*}

\begin{subrem}
Let $k$ be a perfect field of characteristic $p$.
Kerz and Strunk have shown in \cite{KS} that $K_n(S)$
is a $p$-primary torsion group for $n<-d$.
Then Theorem \ref{vanishing} holds for $k$ up to $p$-torsion.

If in addition $k$ is a perfect field, over which
weak resolution of singularities holds, then 
Theorem \ref{vanishing}(1,2) holds for $k$.
This also follows from \cite{KS}; if strong resolution of 
singularities holds, (1) also follows from the Geisser--Hesselholt
theorem in \cite{GH} that $K_n(S)=0$ for $n<-d$.
%This is because if $S$ is a $d$-dimensional scheme, 
%essentially of finite type over $k$,
\end{subrem}

When $S$ is a curve, not necessarily defined over $\mathbb Q$,
we have a similar result.

\begin{theorem}\label{1dim}
Let $f:X\to S$ be a finite map of 1-dimensional noetherian schemes.
Then $K_{-1}(f)$ fits into an exact sequence
\[
0 \to H_\nis^1(S,f_*\cO_X^\times/\cO_S^\times) \to K_{-1}(f) \to 
H_\nis^0(S,f_*\Z/\Z) \to 0.
\]
In addition, $K_{-2}(f)\cong H_\nis^{1}(S, f_{*}\Z/\Z)$ and $K_n(f)=0$ for $n<-2$.
\end{theorem}

\begin{comment}
\begin{proof}
We may suppose that $S$ (and hence $X$) is affine.  By \cite[2.8]{W-AI},
$S$ and $X$ are $K_n$-regular for all $n<0$ and $K_n(S)=K_n(X)=0$ for $n<-1$.
By \cite[6.1]{wei}, $K_{-1}(S)=H^1_\zar(S,\Z)=H^1_\nis(S,\Z)$, and 
similarly for $K_{-1}(X)$. (This is implicit in \cite[XII.10]{b}.)
Now the proof of Theorem \ref{vanishing} goes through.
\end{proof}
\end{comment}

\begin{proof}%[Alternate proof]
By Thomason-Trobaugh \cite[10.8]{TT}, %Theorem 2.8 of \cite{Has},
we have a spectral sequence 
\[
E_{2}^{p,q}= H_\nis^{p}(S, \cK_{-q}(f))\Rightarrow K_{-p-q}(f),
\]
where $\cK_{n}(f)$ is the Nisnevich sheafification of the
presheaf $U\mapsto K_{n}(U, f^{-1}U)$. Each stalk $\cK_{n}(f)$ is 
$K_{n}(A, B)$, where $A$ is a hensel local ring of dimension $\le1$. 
By Lemma \ref{hensel-1dim}, we have
\[
\cK_{n}(f)= \begin{cases} 0 & {\rm if}~ n\leq -2 \\
                      f_{*}\Z/\Z & {\rm if}~ n=-1\\
                      f_{*}\cO_X^\times/\cO_S^\times & {\rm if}~ n=0.
  \end{cases}
\]
Since $cd_\nis(S)\le1$, $E_{2}^{p,q}\neq 0$ only for $p=0, 1$ and $q\le1$.
Thus the spectral sequence degenerates to yield
$K_{-2}(f)\cong H_\nis^1(S,f_*\Z/Z)$ and $K_n(f)=0$ for $n<-2$.
\end{proof}

\begin{lemma}\label{hensel-1dim}
Let $A$ be a 1-dimensional hensel local ring and $f:A\hookrightarrow B$
a finite extension. If $B$ has $r$ components, then
\[
K_0(f)\cong B^\times/A^\times, \quad K_{-1}(f)\cong \Z^{r-1} \quad
\textrm{and\quad $K_n(f)=0$ for $n<-1$.}
\]
%$K_0(f)\cong B^\times/A^\times$, $K_{-1}(f)\cong \Z^{r-1}$
%and $K_n(f)=0$ for $n<-1$.
\end{lemma}

\begin{proof}
Since $B$ is a finite $A$-algebra, $B$ is a finite
product of $r$ hensel local rings.
By \cite[2.8]{W-AI}, $K_n(A)=K_n(B)=0$ for $n<-1$.
By a result of Drinfeld \cite[III.4.4.3]{WK}, we have
$K_{-1}(A)=K_{-1}(B)=0$. The result now follows from \eqref{kgroup}.
\end{proof}

% Note that $E_{2}^{p,q}\neq 0$ only when $p=0, 1$ and $q< 2$ (because
% the Nisnevich cohomological dimension of $S$ is at most $\dim
% S$). So the descent spectral sequence gives us an exact sequence
%$$E_{2}^{0,1}\to E_{2}^{2,0}\to K_{-2}(f)\to E_{2}^{1,1}\to E_{2}^{3,0}.$$
%Therefore,  $K_{-2}(f)\cong E_{2}^{1,1}=H_\nis^{1}(S, \cK_{-1}^{f})=
% H_\nis^{1}(S, f_{*}\Z/\Z).$

\begin{remark}
A necessary condition for $K_{-1}(f)=0$ is that 
the ring extension $f:A\hookrightarrow B$ is {\it anodal}, i.e.,
%a ring extension $f:A\hookrightarrow B$ is 
%called {\it anodal} 
if every $b\in B$ such that $(b^{2}- b)\in A$ and 
$(b^{3}- b^{2})\in A$ belongs to $A$. (See \cite[3.1]{wei}.)
This is because \eqref{eq:det} induces a surjection 
$L\det:K_{-1}(f)\to L\Pic(f)$, and we showed in \cite{SW} that 
$L\Pic(f)=0$ implies that $A\subset B$ is anodal.
%Because \eqref{eq:det} induces a surjection $L\det:K_{-1}(f)\to L\cI(f)$, 
%we see that
%$K_{-1}(f)=0$ implies that $f$ is anodal. 
The converse does not hold, even if $f$ is a birational extension of 
domains, as Example 3.5 in \cite{wei} shows.
\end{remark}

\begin{comment}
 \begin{remark}
Suppose that $S$ is a 2-dimensional excellent noetherian scheme. By
   Theorem 4.4 of \cite{wei duke}, $K_{n}(S)=0$ for $n<-2$. In this
   case, at stalks $\cK_{n}(f)$ is just $K_{n}(A, B)$, where $A$ is
   an excellent hensel local ring. Since $B$ is a finite $A$-algebra,
   $B$ also excellent with $\dim B\leq 2$. So we get $K_{n}(A, B)=0$
   for $n< -3$ and the sequence
  $$ 0\to K_{-2}(A, B)\to K_{-2}(A)\to K_{-2}(B)\to K_{-3}(A, B)\to 0$$
   is exact. Finally, we get $\cK_{n}(f)=0$ for $n<-3$ and an exact
sequence of Nisnevich sheaves on $S$: 
$$
0\to \cK_{-2}(f)\to \cK_{-2}\to f_{*}\cK_{-2}\to \cK_{-3}(f)\to 0.
$$
\end{remark}
\end{comment}

\begin{example}
Here is an example to show why we assume $S$ affine in 
Proposition \ref{subint}.
For each $n$, the scheme $S=\mathbb{P}^1_k$ has a sheaf of algebras 
$\cO_B=\cO_S\oplus\cO(n)$ with $\cO(n)$ a square-zero ideal; fix $n\le-2$
and set $X=\Spec(\cO_B)$. Then  $H=H^1(\mathbb{P}^1,\cO(n))$ is nonzero
and $\Pic(X)=\Pic(S)\oplus H$, $K_0(X)\cong K_0(S)\oplus H$.
In particular, $K_{-1}(f)=H\ne0$.
\end{example}

\medskip
\textbf{Acknowledgements}: 
This project was initiated while the first author was visiting Rutgers 
University in August 2015; he would like to thank the Math Dept.\ of
Rutgers University for the invitation and financial support.
The first author is also grateful to Jan Stienstra
for sending him the manuscript \cite{Jan}.  The second author would like to
thank TIFR for providing a great environment for doing this research.

\end{document}